\documentclass{amsart}
\usepackage[utf8]{inputenc}
\usepackage{amsmath,amssymb,amsthm}
\usepackage{tikz}
\usepackage{subcaption}
\usepackage{mathtools,eucal}
\usepackage{setspace}
\usepackage[english]{babel}
\usepackage[pdfauthor={Suraj Krishna},pdftitle={Vertex links and the Grushko decomposition},pdftex]{hyperref}
\hypersetup{colorlinks,%
citecolor=black,%
filecolor=black,%
linkcolor=blue,%
urlcolor=black,%
}
\usepackage[capitalize,nameinlink,noabbrev]{cleveref}

\newtheorem{theorem}{Theorem}[section]
\newtheorem{prop}[theorem]{Proposition}
\newtheorem{lemma}[theorem]{Lemma}
\newtheorem{cor}[theorem]{Corollary}

\newtheorem{thm}{Theorem}

\newtheorem{corr}[thm]{Corollary}

\theoremstyle{definition}
\newtheorem{defn}[theorem]{Definition}
\usepackage{microtype}
\theoremstyle{remark}
\newtheorem{remark}[theorem]{Remark}
\newtheorem*{convention}{Convention}
\newtheorem*{notation}{Notation}

\newcommand{\cat}{\mathrm{CAT(0)}}
\newcommand{\wt}{\widetilde}
\newcommand{\st}{\mathrm{star}}
\newcommand{\VH}{\mathcal{VH}}
\DeclareMathOperator{\link}{link}

\usepackage{enumitem}

\title{Vertex links and the Grushko decomposition}
\author{Suraj Krishna M S}
\date{\today}

\address{School of Mathematics, Tata Institute of Fundamental Research, Mumbai 400005, India}

\email{suraj@math.tifr.res.in}

\begin{document}

\maketitle
\begin{abstract}
We develop an algorithm of polynomial time complexity to construct the Grushko decomposition of fundamental groups of graphs of free groups with cyclic edge groups. Our methods rely on analysing vertex links of certain CAT(0) square complexes naturally associated with a special class of the above groups. 
Our main result transforms a one-ended CAT(0) square complex of the above type to one whose vertex links satisfy a strong connectivity condition, as first studied by Brady and Meier.
\end{abstract}

\section{Introduction}

The Grushko decomposition theorem \cite{grushko} states that a finitely generated group is a free product of finitely many freely indecomposable non-free groups and a finite rank free group. This decomposition is unique in the sense that the freely indecomposable groups appearing in this decomposition are unique up to reordering and conjugation, and the rank of the free group is invariant.

Given a finitely presented group, there is no algorithm in general to compute its Grushko decomposition. Even when algorithms do exist, they are often not tractable. In this article, we develop an algorithm, with polynomial time complexity, to compute the Grushko decomposition of fundamental groups of graphs of free groups with cyclic edge groups (see \cref{thm_grushko_general_intro}). 

Our algorithm is obtained as a consequence of \cref{cor_grushko_intro}, which gives an algorithm to obtain the Grushko decomposition of fundamental groups of compact nonpositively curved square complexes that we call \emph{tubular graphs of graphs}.
A tubular graph of graphs (see \cref{defn_tubular_graph_graphs} for the precise definition) is a square complex obtained by attaching finitely many tubes (a \emph{tube} is the Cartesian product of a circle and the unit interval) to a finite collection of finite graphs. Tubular graphs of graphs are thus nonpositively curved $\VH$-complexes (in the sense of Wise \cite{wisephd}) in which vertical hyperplanes are homeomorphic to circles (see \cref{section_setup}). 

In \cite{stallings_whitehead_graphs}, Stallings built on work by Whitehead \cite{whitehead} and developed an algorithm that takes a free group and finitely many cyclic subgroups as input and decides whether or not the free group splits freely relative to the cyclic subgroups. We use the machinery of tubular graphs of graphs in this article to develop an alternative algorithm.

In fact, we also obtain an analogue for tubular graphs of graphs of a result by Jaco \cite{jaco_free_product} in 3-manifold theory which states that if the fundamental group of a 3-manifold splits as a free product, then each free factor is itself the fundamental group of a 3-manifold.

Our approach is geometric and in fact we are more interested in nonpositively curved square complexes than groups.
We show how to construct the Grushko decomposition by cutting along contractible subspaces of tubular graphs of graphs which induce free splittings. 

Crucial to our methods is the following key result by Brady and Meier:

\begin{theorem}[\cite{bradymeier}] \label{prop_bradymeier}
Let $X$ be a finite connected nonpositively curved cube complex. Suppose that 
\begin{enumerate}[label=\text{(BM\arabic*)}]
\item \label{BM1} for each vertex $v \in X$, the link of v is connected and
\item  \label{BM2} for each vertex $v \in X$ and each simplex $\sigma$ in $\link(v)$, $\link(v) \setminus \sigma$ is (non-empty and) connected. 
\end{enumerate}
Then $\widetilde{X}$ is one-ended.
\end{theorem}  

We say that a cube complex is \emph{Brady-Meier} if it satisfies the conditions \ref{BM1} and \ref{BM2} above. Note that the Brady-Meier conditions are local and hence preserved by covering maps.
The converse of \cref{prop_bradymeier} is not true in general. The main result of the article gives a geometric/combinatorial procedure that modifies a given tubular graph of graphs to a homotopy equivalent tubular graph of graphs which is Brady-Meier if and only if the fundamental group is one-ended:

\begin{thm}[\cref{thm_ends_main}] \label{thm_ends_main_intro}
There is an algorithm of polynomial time complexity which takes a tubular graph of graphs as input and returns a homotopy equivalent tubular graph of graphs which is either a Brady-Meier complex or contains a locally disconnecting vertex which splits the fundamental group as a free product.
\end{thm} 

For this and the corollaries below, the running times of the algorithms are polynomial in the number of squares of the input square complex. A major application of the above theorem is in the development of an algorithm to obtain the JSJ decomposition of one-ended hyperbolic fundamental groups of tubular graphs of graphs \cite{meda_jsj}. The main point is that one can work with a Brady-Meier tubular graph of graphs (obtained in polynomial time) when the fundamental group is one-ended, thanks to \cref{thm_ends_main_intro}.

The key step in the construction of our algorithm involves a simplification of the input tubular graph of graphs by `opening-up' at a vertex which does not satisfy \ref{BM2}. This opening-up keeps the number of squares in the complex constant, while simplifying certain vertex links. We call such an opening procedure as an \emph{SL-move} (`SL' stands for simplified link). We thus obtain a partial converse to \cref{prop_bradymeier}:

\begin{corr}\label{thm_partial_converse_brady_meier}
A tubular graph of graphs has a one-ended universal cover if and only if it can be simplified in finitely many SL-moves to a Brady-Meier tubular graph of graphs with isomorphic fundamental group.
\end{corr}

We also obtain the Grushko decomposition of these groups using \cref{thm_ends_main_intro}. In fact, we obtain a stronger result:

\begin{corr} [\cref{thm_actual_ends_algo}] \label{cor_grushko_intro}
There is an algorithm of polynomial time complexity which takes as input a tubular graph of graphs and returns a homotopy equivalent tubular graph of graphs obtained by gluing together certain vertices of a finite collection of Brady-Meier tubular graphs of graphs and finite graphs. 
Further, the free product decomposition induced by cutting along the glued vertices is the Grushko decomposition of the fundamental group of the input tubular graph of graphs.  
\end{corr}

We point out that our proof neither uses Stallings' theorem on ends of finitely generated groups nor assumes the existence of a Grushko decomposition. In fact, our procedure yields a new proof of Stallings' theorem for fundamental groups of tubular graphs of graphs as well as the existence of a Grushko decomposition for these groups.

The analogue of Jaco's result immediately follows:

\begin{corr}[\cref{cor_subadditivity_sphere_theorem}]\label{cor_subadditivity_intro}
Let $X$ be a tubular graph of graphs with fundamental group $G$. If $G = A * B$, then there exist tubular graphs of graphs $X_1$ and $X_2$ such that $A$ and $B$ are fundamental groups of $X_1$ and $X_2$ respectively. Moreover, $X_1$ and $X_2$ can be so chosen such that the total number of squares in $X_1$ and $X_2$ is bounded by the number of squares in $X$.
\end{corr}

The Grushko decomposition may be found algorithmically in other situations. Jaco, Letscher and Rubinstein \cite{jaco_grushko_jsj_algorithm} gave an algorithm of polynomial time complexity to compute the prime decomposition of a 3-manifold from a triangulation. 
Gerasimov \cite{gerasimov} showed that the Grushko decomposition can be computed for hyperbolic groups. But there is no known bound on the complexity of his algorithm. 
Dahmani and Groves \cite{dahmani_groves_grushko} extended Gerasimov's ideas to groups which are hyperbolic relative to abelian subgroups. 
Diao and Feighn \cite{feighn_grushko} gave an algorithm for graphs of free groups.
Their algorithm relies on certain simplifications of a given graph of free groups which depend on finding Gersten representatives of the incident edge groups in each vertex group, which is algorithmic. More recently, Touikan \cite{touikan_two_torsion} presented an important algorithm which returns the Grushko decomposition of finitely presented groups with solvable word problem and no 2-torsion. However, the time complexity of his algorithm is not known. 

We note that our algorithm is explicit and neither uses the Rips machine nor requires solutions of equations in free groups, as some of the above algorithms do. 

We mention another application of our algorithm, before stating our result for general graphs of free groups with cyclic edge groups. 
As defined by Stallings in \cite{stallings_whitehead_graphs}, a finite set of words $W$ of a finite rank free group $F$ is \textit{separable} if there exists a nontrivial free splitting of $F$ such that each word of $W$ conjugates into a free factor.

Stallings obtained an algorithm to detect separability in \cite{stallings_whitehead_graphs}. In a related result, Roig, Ventura and Weil \cite{whitehead_minimization} obtained a polynomial time algorithm to solve the Whitehead minimization problem and therefore the primitivity problem.
We give an alternate version of Stallings' algorithm using \cref{thm_ends_main_intro}:
\begin{corr}[\cref{thm_separability_algo}] \label{cor_separability_intro}
There exists an algorithm of polynomial time complexity that takes a finite set of words in a finite rank free group as input and decides whether it is separable.
\end{corr}

Stallings obtains his algorithm by constructing a Whitehead graph for the given set of words in a chosen basis. He then uses a Whitehead automorphism to modify the basis whenever there is a cut vertex in the Whitehead graph to reduce the total length of the given set of words. We first give a new proof of Whitehead's cut vertex theorem (\cref{prop_stallings_separability}). We then obtain our algorithm by first constructing the tubular graph of graphs associated to a `double' of the free group with the given set of words. We then apply the algorithm of \cref{thm_ends_main_intro}. 

By combining a result of Wilton \cite{wilton_one-ended_groups} (see \cref{thm_wilton_ends}) with \cref{cor_grushko_intro}, we obtain the advertised result:

\begin{thm}[\cref{thm_grushko_general}] \label{thm_grushko_general_intro}
There exists an algorithm of polynomial time complexity which takes a graph of free groups with cyclic edge groups as input and returns as output the Grushko decomposition of its fundamental group.
\end{thm}
The running time of the algorithm is polynomial in the sum of the lengths of the words induced by the generators of the edge groups in the respective vertex groups.

\textit{Acknowledgements.} I thank my advisors Fr\'ed\'eric Haglund and Thomas Delzant for suggesting the problem and for several important discussions. Part of this work was done when I was visiting IRMA, Strasbourg. This work was supported by a French public grant for research, as part of the Investissement d'avenir project, reference ANR-11-LABX-0056-LMH, LabEx LMH. I am grateful to the referee for their suggestions to improve the paper.

\section{The setup} \label{section_setup}

\subsection{\texorpdfstring{$\VH$}{VH}-complexes}

The notion of $\VH$-complexes was first introduced in \cite{wisephd}.

\begin{defn} A \textit{square complex} is a two dimensional CW complex in which each 2-cell is attached to a combinatorial loop of length 4 and is isometric to the standard Euclidean unit square $I^2 = [0,1]^2$.
\end{defn} 

All our square complexes will be locally finite.

\begin{defn}[Vertex links] Let $v\in X$ be a vertex of a square complex. The link of $v$, denoted by $\link(v)$ is a graph whose vertex set is the set $\{e \mid e \mathrm{\: is \: a \: half\mbox{-}edge \: incident \: to \:} v\}$. The number of edges between two vertices $e, f$ is the number of squares of $X$ in which $e,f$ are adjacent half-edges.
\end{defn}

\begin{defn} A square complex is \emph{nonpositively curved} if the length of a closed path in the link of any of its vertices is at least four.
\end{defn}

By a result of Gromov \cite{gromov_hyperbolic}, a simply connected nonpositively curved square complex is $\cat$ in the metric sense.

\begin{defn}[\cite{sageev}]
Let $X$ be a square complex. A \emph{mid-edge} of a square $\mathsf{s}$ in $X$ is an edge (after subdivision of $\mathsf{s}$) running through the center of $\mathsf{s}$ and parallel to two of the edges of $\mathsf{s}$.
Declare two edges $e$ and $f$ to be equivalent if there exists a sequence $e = e_1, \cdots, e_n = f$ of edges such that $e_i$ and $e_{i+1}$ are opposite edges of some square of $X$.
Given an equivalence class $[e]$ of edges, the \emph{hyperplane dual to $e$}, denoted by $\mathsf{h}_e$, is the collection of mid-edges which intersect edges in $[e]$.
\end{defn}

\begin{defn}[\cite{wisephd}] A \textit{$\VH$-complex} is a square complex in which every 1-cell is labelled as either vertical or horizontal in such a way that each 2-cell is attached to a loop which alternates between horizontal and vertical 1-cells.
\end{defn}

The labelling of the edges of a $\VH$-complex as horizontal and vertical induces a labelling of the vertices in the link of any vertex as horizontal and vertical, thus making the link a bipartite graph. Similarly, the hyperplanes of a $\VH$ complex are also labelled as vertical and horizontal, with a vertical hyperplane being dual to an equivalence class of horizontal edges and a horizontal hyperplane being dual to an equivalence class of vertical edges.

\begin{remark}
Since the link of any vertex of a $\VH$-complex is bipartite, the length of a closed path is even. Thus a $\VH$-complex is nonpositively curved if there exists no bigon in any vertex link. 
\end{remark}

\subsection{Graphs of spaces} \label{section_graphs_spaces}
Graphs of groups are the basic objects of study in Bass-Serre theory \cite{serre}.
It was studied from a topological perspective in \cite{scottandwall} by looking at graphs of spaces instead of graphs of groups. We will adopt this point of view.

\begin{defn} \label{defn_graph_spaces} By a \textit{graph of spaces}, we mean the following data: $\Gamma$ is a connected graph, called the underlying graph. For each vertex $s$ (edge $a$) of $\Gamma$, $X_s$ ($X_a$) is a topological space. Further, whenever $a$ is incident to $s$, $\partial_{a,s} : X_{a} \to X_{s}$ is a $\pi_1$-injective continuous map. The \textit{geometric realisation} of the above graph of spaces is the space $X =( \bigsqcup_{s \in \Gamma^{(0)}} X_{s} \sqcup \bigsqcup_{a \in \Gamma^{(1)}} X_{a} \times [0,1]) /\sim$, where $(x,0)$ and $(x,1)$ are identified respectively with $\partial_{a,s}(x)$ and $\partial_{a,s'}(x)$. Here, $s$ and $s'$ are the two endpoints of $a$.
\end{defn}

Note that the universal cover of $X$ has the structure of a \emph{tree of spaces}, a graph of spaces whose underlying graph is the Bass-Serre tree of the associated graph of groups structure of $X$ \cite{scottandwall}.

\subsection{Tubular graphs of graphs}

\begin{defn} \label{defn_tubular_graph_graphs}
A \textit{tubular graph of graphs} is a finite graph of spaces in which each vertex space is a finite connected simplicial graph and each edge space is a simplicial graph homeomorphic to a circle. Further, the attaching maps are simplicial immersions. We will always assume that the underlying graph is connected.
\end{defn}
We note that no vertex graph is a tree, as a consequence of the definition. We also remark that asking for each vertex graph to be simplicial is not a serious restriction as every one dimensional CW complex is a simplicial graph after subdivision.

\begin{remark}
Tubular graphs of graphs are topological versions of certain graphs of free groups with cyclic edge groups (see \cref{section_general} for a definition). We note that any graph of free groups with cyclic edge groups in which the underlying graph is a tree can be realised topologically as a tubular graph of graphs. If the underlying graph contains a loop such that the generator of the edge group is attached to words of different lengths in the incident vertex group, then such a graph of free groups with cyclic edge groups cannot be realised as a tubular graph of graphs. In particular, this rules out non-Euclidean Baumslag-Solitar groups.
\end{remark}
We also have
\begin{prop}[\cite{wisephd}] \label{prop_tubular_vh}
The geometric realisation of a tubular graph of graphs is a finite (hence compact), connected nonpositively curved $\VH$-complex whose vertical hyperplanes are circles.
\end{prop}

\begin{convention} Throughout this text, we will use the same notation for a graph of graphs and the $\VH$-complex which is its geometric realisation. $X$ will denote a tubular graph of graphs with underlying graph $\Gamma_X$. Let $s \in \Gamma_X$ be a vertex. Then $X_s$ will denote the vertex graph at $s$ and if $a$ is an edge of $\Gamma_X$, we will denote the edge graph at $a$ by $X_{a}$. Thus, every edge of any vertex graph $X_s$ will be a vertical edge in the $\VH$-complex $X$ while horizontal edges in $X$ are the edges of the form $\{v\} \times [0,1]$, for vertices $v$ in the edge graphs $X_a$.
\end{convention}

\begin{defn} [Thickness] \label{def_thickness}
For an edge $e$ in $X$,  the \textit{thickness} of $e$ is the number of squares of $X$ which contain $e$.
\end{defn}

Observe that a horizontal edge of $X$ always has thickness equal to two.

\begin{defn} Let $X_{s}$ be a vertex graph of a tubular graph of graphs $X$. We say that $X_{s}$ (and hence $X$) has a \textit{hanging tree} if $X_{s}$ is a wedge of two subgraphs $A$ and $B$ such that one of them, say $A$, is a tree. Here, $A$ is called a hanging (sub)tree of $X_{s}$. 

\end{defn}

\begin{remark} Since the attaching maps of edge graphs are immersions, an edge in a hanging tree of $X$ has thickness zero.
\end{remark}

We thus have that

\begin{lemma}
A tubular graph of graphs is homotopy equivalent to a tubular graph of graphs with no hanging trees.
\end{lemma}

\begin{defn} An edge $e$ in $X_s$ is a \emph{rudimentary edge} if it is of thickness one and moreover $X_{s}$ is a circle.
\end{defn}

\begin{lemma}\label{lemma_removing_rudimentary_edges}
A tubular graph of graphs is homotopically equivalent to a tubular graph of graphs with no rudimentary edges.
\end{lemma}
\begin{proof} 
Let $e$ be a rudimentary graph in a vertex graph $X_{s}$ of a tubular graph of graphs $X$. Since $X_s$ is a circle and attaching maps of edge graphs to $X_s$ are graph immersions, there exists exactly one edge $a$ incident to $s$ in the underlying graph $\Gamma_X$ and the attaching map from $X_a$ to $X_s$ is a graph isomorphism (see \cref{fig: rudimentary edge}).

\begin{figure}
\begin{center}
\begin{tikzpicture}[scale = 0.5]
\draw (0,0).. controls (1,1) and (1,3).. (0,4) ;
\draw [dotted] (0,0).. controls (-0.7,1) and (-.7,3).. (0,4) ;
\shade[gray] (0.7,1.5) rectangle (2.7,2.5);
\draw (0.7,1.5) -- (0.7,2.5) ;
\draw (0,0) -- (2.7,0) ;
\draw (0,4) -- (2.7,4) ;
\draw [->] (0.7,3) -- (2.6,3) ;
\draw [->] (0.7,1) -- (2.6,1) ;
\filldraw (0.7,1.5) circle (1pt) ;
\filldraw (0.7,2.5) circle (1pt) ;
\draw (0.7,1.5) -- (2.7,1.5) ;
\draw (0.7,2.5) -- (2.7,2.5) ;
\filldraw (2.7,1.5) circle (.3pt) ;
\filldraw (2.7,2.5) circle (.3pt) ;
\filldraw (0.7,1.5) circle (1pt) ;
\filldraw (0.7,2.5) circle (1pt) ;
\filldraw (1.85,1.5) circle (0.5pt) ;
\filldraw (1.85,2.5) circle (0.5pt) ;
\draw [dotted] (1.5,0).. controls (2,1) and (2,3).. (1.5,4) ;
\draw [dotted] (1.5,0).. controls (1.1,1) and (1.1,3).. (1.5,4) ;
\node [below] at (0,0) {$X_{s}$} ;
\node [below] at (1.5,0) {$X_{a}$} ;
\node [left] at (0.7,2) {$e$} ;
\end{tikzpicture}
\end{center}
\caption{Removing rudimentary edges.} \label{fig: rudimentary edge}
\end{figure}
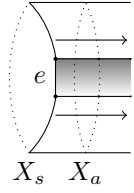
Thus $X$ is homotopic to $X'$ obtained by removing $X_a$ and the open tube containing $X_a$. $\Gamma_{X'}$ is the graph obtained from $\Gamma_X$ by collapsing $a = (s, s')$ to $s'$. Repeating this procedure at each rudimentary edge gives the result.
\end{proof}

\subsection{Ends}
The theory of ends of a topological space was first studied by Freudenthal \cite{Freudenthal_ends}. 
The notion we require is that of ``one-endedness''. We will use the following definition due to Specker (see \cite{specker} or \cite{raymond}).

\begin{defn}
A locally finite CW complex $X$ is \emph{one-ended} if for every compact set $K$, $X \setminus K$ has exactly one unbounded component.
\end{defn}

It is a well-known fact that being one-ended is a quasi-isometry invariant (see Proposition I.8.29 of \cite{bridsonhaefliger}, for instance). Then by an application of the \v Svarc-Milnor Lemma (see Proposition I.8.19 of \cite{bridsonhaefliger}), for instance), we have the following definition of one-endedness of a finitely presented group.

\begin{prop} Let $G$ be a finitely presented group and $X$ be a finite connected CW complex such that $G \cong \pi_1(X)$. $G$ is one-ended if and only if $\wt{X}$ is one-ended. 
\end{prop}

\section{Not one-ended}
In this section, we will collect a few results that help determine when the fundamental group of a tubular graph of graphs is not one-ended.

\begin{defn} Let $Z$ be a CW complex and $v \in Z$ be a vertex. Let $\tilde{v}$ denote a lift of $v$ in the universal cover $\wt{Z}$ of $Z$. The \emph{star of $v$}, denoted by $\st(v)$, is the smallest subcomplex of $\wt{Z}$ which contains all cells $\sigma$ such that $\tilde{v} \in \sigma$.  The \emph{open star of $v$}, denoted by $\mathring{\st}(v)$, is the interior of $\st(v)$. 
\end{defn}

We first recall a classical result due to Hopf:

\begin{lemma}[\cite{hopf_groupends}] \label{cor_freee_amalgam_ends}
Let $G$ be a torsion-free finitely generated group such that either \begin{enumerate}
\item $G = H * K$ is a nontrivial free splitting of $G$, or 
\item  $G = H *_{1}$ is an HNN extension of a finitely generated group $H$ over its trivial subgroup.
\end{enumerate}
Then $G$ is not one-ended.
\end{lemma}

We remark that the two conditions above are cases of an edge of groups with trivial edge group. We make a distinction between them because of the following standard lemma.

\begin{lemma} \label{lemma_free_amalgam_hnn} Let $Z$ be a connected CW complex. Let $c \in Z$ be a vertex. Suppose that either
\begin{enumerate}
\item $Z = Z_1 \vee_{\{c\}} Z_2$, with neither $Z_1$ nor $Z_2$ simply connected, or
\item $\st(c) \setminus \{c\}$ is not connected, but $Z\setminus \{c\}$ is connected.
\end{enumerate}
Then $\widetilde{Z}$ is not one-ended.
\end{lemma}

\begin{prop} \label{prop_thickness0} Let $X$ be a tubular graph of graphs with no hanging trees. Suppose there exists an edge of thickness zero. Then $\widetilde{X}$ is not one-ended.
\end{prop}
\begin{proof}
Since all horizontal edges have thickness two, an edge of thickness zero has to be vertical. Let $e$ in $X_{s}$ be such an edge with midpoint $c$. Subdivide $e$ so that $c$ is a vertex of $X$. Either $X \setminus \{c\}$ is connected, or $X = X_1 \vee_c X_2$ is a wedge of two subcomplexes. Let $X_{s} = A \vee_c B$ be the induced decomposition of $X_{s}$. Since $e$ is not in a hanging tree, neither $A$ nor $B$ is a tree. Thus, $X_1$ and $X_2$ are not simply connected (as $\pi_1(A) \hookrightarrow \pi_1(X_1)$ and $\pi_1(B) \hookrightarrow \pi_1(X_2)$ in the graph of groups setup \cite{serre}). \cref{lemma_free_amalgam_hnn} then gives the result.
\end{proof}

\begin{prop} \label{prop_thickness1} Let $X$ be a tubular graph of graphs with no hanging trees and no rudimentary edges. Suppose there exists an edge of thickness one. Then $\widetilde{X}$ is not one-ended. \end{prop}

\begin{proof}
Let $e$ be an edge in $X_s \hookrightarrow X$ of thickness one and $\mathsf{s}$ be the lone square containing $e$. We will show that $X$ is homotopic to a wedge of two non-simply connected square complexes.

Note that $X$ is homotopic to a complex obtained by removing the open square $\mathsf{s}$ and the open edge $e$ (\cref{fig: thickness one edge}). Removing $\mathsf{s}$ decreases the thickness of a horizontal edge $f$ adjacent to $e$. We remove the only square that contains $f$, which in turn creates another horizontal edge of thickness one. We continue removing until we end up with a horizontal edge $f'$ (adjacent to $e$, see \cref{fig: thickness one edge}) of thickness zero.

\begin{figure}
\begin{center}
\begin{tikzpicture} [scale = 0.5]
\draw (0,0).. controls (1,1) and (1,3).. (0,4) ;
\draw [dotted] (0,0) -- (-.5,.5) ;
\draw [dotted] (-0.5,3.5) -- (0,4) ;
\shade[gray] (0.7,1.5) rectangle (2.7,2.5);
\draw (0.7,1.5) -- (0.7,2.5) ;
\draw (0,0) -- (2.7,0) ;
\draw (0,4) -- (2.7,4) ;
\draw (0.7,1.5) rectangle (2.7,2.5) ;
\filldraw (0.7,1.5) circle (1pt) ;
\filldraw (0.7,2.5) circle (1pt) ;
\draw (0.7,1.5) -- (2.7,1.5) ;
\draw (0.7,2.5) -- (2.7,2.5) ;
\filldraw (2.7,1.5) circle (.3pt) ;
\filldraw (2.7,2.5) circle (.3pt) ;
\filldraw (0.7,1.5) circle (1pt) ;
\filldraw (0.7,2.5) circle (1pt) ;

\node [below] at (1.85,1.5) {$\mathsf{s}$};
\node [below] at (0,0) {$X_{s}$} ;
\node [left] at (0.7,2) {$e$} ;
\draw [dotted] (2.7,1.5) -- (3,1) ;
\draw [dotted] (2.7, 2.5) -- (3,3) ;
\draw [->] (3.5,2) -- (4.5,2) ;
\node [below] at (1.5,-1) {{\large $X$}} ;
\draw (6,0).. controls (6.5,.3) and (6.8,1).. (6.7,1.5) ;
\draw (6.7,2.5).. controls (6.8,3.5) and (6.2,3.8).. (6,4) ;
\draw [dotted] (6,0) -- (5.5,.5) ;
\draw [dotted] (5.5,3.5) -- (6,4) ;
\draw (8.7,1.5) -- (8.7,2.5) ;
\draw (6,0) -- (8.7,0) ;
\draw (6,4) -- (8.7,4) ;
\filldraw (6.7,1.5) circle (1pt) ;
\filldraw (6.7,2.5) circle (1pt) ;
\draw (6.7,1.5) -- (8.7,1.5) ;
\draw (6.7,2.5) -- (8.7,2.5) ;
\filldraw (8.7,1.5) circle (.3pt) ;
\filldraw (8.7,2.5) circle (.3pt) ;
\draw [dotted] (8.7,1.5) -- (9,1) ;
\draw [dotted] (8.7, 2.5) -- (9,3) ;

\draw [->] (9.5,2) -- (10.5,2) ;

\node [above] at (8.3,2.5) {$f$} ;
\draw (12,0).. controls (12.5,.3) and (12.8,1).. (12.7,1.5) ;
\draw (12.7,2.5).. controls (12.8,3.5) and (12.2,3.8).. (12,4) ;
\draw [dotted] (12,0) -- (11.5,.5) ;
\draw [dotted] (11.5,3.5) -- (12,4) ;
\draw (14.7,1.5) -- (14.7,2.5) ;
\filldraw (12.7,1.5) circle (1pt) ;
\filldraw (12.7,2.5) circle (1pt) ;
\draw (12.7,1.5) -- (14.7,1.5) ;
\filldraw (14.7,1.5) circle (.3pt) ;
\filldraw (14.7,2.5) circle (.3pt) ;
\draw [dotted] (14.7,1.5) -- (15,1) ;
\draw [dotted] (14.7, 2.5) -- (15,3) ;
\node [below] at (11.5,2) {$X_1$};
\node [below] at (16,2) {($X_2$)};
\node  at (13.8,2) {$f'$};
\node [below] at (13.8,-1) {{\large $X'$}};
\end{tikzpicture}
\caption{Removing squares containing thickness-one edges.} \label{fig: thickness one edge}
\end{center}   
\end{figure}
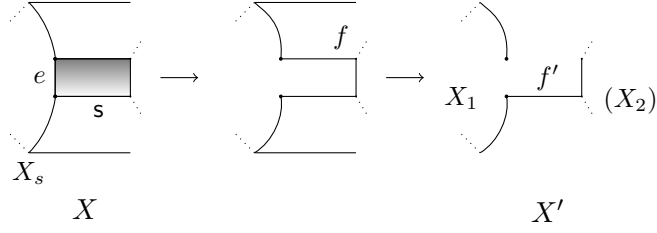

Call the resulting subcomplex as $X'$. Either the midpoint of $f'$ does not disconnect $X'$, or $X'$ is a wedge of two subcomplexes $X_1$ and $X_2$, say (\cref{fig: thickness one edge}). Note that $X_1$ is not simply connected as $X$ had neither hanging trees nor rudimentary edges. Also $X_2$ is not simply connected as $X_2$ is a subcomplex of $X$ which is not simply connected. By \cref{lemma_free_amalgam_hnn}, $\widetilde{X}$ is not one-ended.
\end{proof}

\begin{prop} \label{prop_1ended_is_BM1} Suppose that every edge of $X$ has thickness  at least two. If $\wt{X}$ is one-ended, then the link of every vertex of $X$ is connected.
\end{prop}

We will need the following result. 

\begin{lemma}\label{lemma_thickness_two_infinite_pi1}
Let $Z$ be a compact, connected nonpositively curved square complex which has at least one edge. If each edge of  $Z$  is contained in at least two squares, then $\pi_1(Z)$ is infinite. 
\end{lemma}

\begin{proof}
Let $e$ be an edge of $Z$ . The hyperplane $\mathsf{h}_e$ dual to $e$ is a finite connected graph in which every vertex is of valence at least two, by the hypothesis on  $Z$ . This implies that $\pi_1(\mathsf{h}_e)$ is a free group of positive rank. 
Any lift of $\mathsf{h}_e$ embeds as a hyperplane in $\widetilde{Z}$, since $\wt{Z}$ is $\cat$ (Theorem 4.10 of \cite{sageev}). This implies that $\pi_1(\mathsf{h}_e) \hookrightarrow \pi_1(Z)$. Hence the result.
\end{proof}

\begin{proof}[Proof of \cref{prop_1ended_is_BM1}]
Let $u \in X_{s}$ be a vertex whose link is not connected. This implies that $\st(u) \setminus \{u\}$ is not connected. The result then follows from \cref{lemma_free_amalgam_hnn}. Indeed, if $X = X_1 \vee_u X_2$, then $X_i$ is not simply connected by \cref{lemma_thickness_two_infinite_pi1}.
\end{proof}

\section{The second Brady-Meier criterion}
In this section, we will assume that each edge of $X$ has thickness at least two and every vertex link is connected, but $X$ does not satisfy the second Brady-Meier criterion \ref{BM2}. We will explain how to simplify $X$ in this case by defining an opening of the complex at a vertex whose link does not satisfy \ref{BM2}. Fix one such vertex $u \in X_s \subset X$.

A vertex of $\link(u)$ is \emph{vertical (horizontal)} if it is a vertical (horizontal) half-edge incident to $u$ in $X$. 
Observe that the horizontal vertices have valence exactly two.

\begin{lemma}\label{lemma_bm2_at_superdivision}
$X$ does not satisfy \ref{BM2} at $u$ if and only if a vertical vertex of $\link(u)$ disconnects $\link(u)$.
\end{lemma}

\begin{proof}
One direction is clear. For the converse, there are two cases to consider: either a horizontal vertex $h$ or an edge $e$ of $\link(u)$ disconnects $\link(u)$.

In the first case, let $v_1$ and $v_2$ be the two vertical vertices adjacent to $h$. Let $C$ be the component of $\link(u) \setminus \{h\}$ that contains $v_1$. Let $x \neq v_1 \in C$ be a vertex. Such a vertex exists as every edge of $X$ has thickness at least two. Then any path in $\link(u)$ from $x$ to $v_2$ meets $h$, and hence meets $v_1$. Thus $v_1$ disconnects $\link(u)$.
In the second case, since horizontal vertices have valence two, it is easy to see that the unique vertical vertex incident on $e$ disconnects $\link(u)$.
 \end{proof}

\subsection*{The opening procedure}
Throughout, we will denote an edge incident to $u$ and the corresponding vertex in $\link(u)$ by the same notation. Let $e$ be a vertical edge incident to $u$ which disconnects $\link(u)$. Let $C_1, \cdots, C_n$ denote the maximal connected subgraphs of $\link(u) \setminus e$ (\cref{fig_link(u)}), where maximality is by inclusion. We will denote by $f_{i1}, \cdots, f_{ik_i}$ the vertical vertices in $C_i$ so that the vertical edges $f_{i1}, \cdots, f_{ik_i}$ incident on $u$ belong to $\st(u,X_s)$ (\cref{fig_opening_X_s}). Let $x_{ij}$ denote the other endpoint of $f_{ij}$.
We will now explain how to open $\st(u,X_s)$:

\begin{figure}
\begin{center}
\begin{tikzpicture} [scale = 1.75]
\filldraw [red] (1,1) circle (2pt) ;
\filldraw [blue] (0,0) circle (1pt) ;
\filldraw [blue] (0.3,-0.2) circle (1pt) ;
\filldraw [blue] (1,-0.3) circle (1pt) ;
\filldraw [blue] (1.3,-0.2) circle (1pt) ;
\filldraw [blue] (1.9,0) circle (1pt) ;
\filldraw [blue] (2.2,0.1) circle (1pt) ;
\draw (1,1) -- (0,0) ;
\draw (1,1) -- (0.3,-0.2) ;
\draw (1,1) -- (1,-0.3) ;
\draw (1,1) -- (1.3,-0.2) ;
\draw (1,1) -- (1.9,0) ;
\draw (1,1) -- (2.2,0.1) ; 
\draw [blue] (0.15,-0.1) circle [radius = 0.45];
\draw [blue] (1.15,-0.25) circle [radius = 0.45];
\draw [blue] (2.05,0.05) circle [radius = 0.45];
\node [above] at (1,1) {$e$}; 
\node [left] at (0,0) {${f_{11}}$};
\node [below] at (0.3,-0.2) {${f_{1k_1}}$};
\node [below] at (1,-0.3) {${f_{21}}$};
\node [below right] at (1.3,-0.2) {${f_{2k_2}}$};
\node [below] at (1.9,0) {${f_{n1}}$};
\node [below right] at (2.2,0.1) {${f_{nk_n}}$};
\draw [dotted, thick] (0.2,0.2) -- (0.4,0) ;
\draw [dotted, thick] (1,-0.1) -- (1.26,-0.1) ;
\draw [dotted, thick] (1.75,0.2) -- (1.95,0.3) ;
\draw [dotted, thick] (1.15,0.5) -- (1.35,0.55) ;
\node [below] at (0.1,-0.6) {$C_1$} ;
\node [below] at (1.15,-0.75) {$C_2$} ;
\node [below] at (2.25,-0.5) {$C_n$} ;
\end{tikzpicture}
\caption{$\link(u)$} \label{fig_link(u)}
\end{center}
\end{figure}

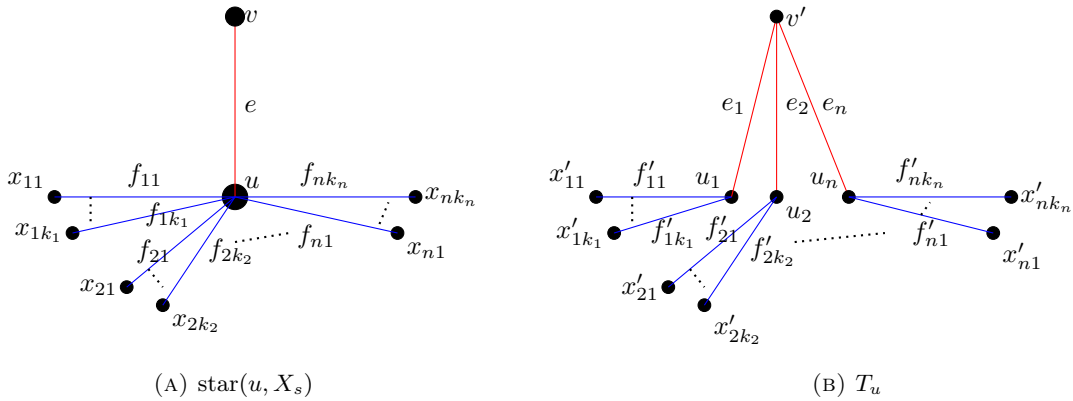
\begin{figure}
\begin{center}
\begin{tikzpicture} [scale = 2]

\begin{scope}[xshift=-3cm]
\filldraw (0,0) circle (2pt) ;
\draw [red] (0,0) -- (0,1) ;
\filldraw (0,1) circle (1.5pt) ;
\filldraw (-1,0) circle (1pt) ;
\filldraw (-0.9,-0.2) circle (1pt) ;
\filldraw (-0.4,-0.6) circle (1pt) ;
\filldraw (-0.6,-0.5) circle (1pt) ;
\filldraw (0.9,-.2) circle (1pt) ;
\filldraw (1,0) circle (1pt) ;
\node [above right] at (0,0) {$u$} ; 
\node [right] at (0,1) {$v$} ; 
\node [above left] at (-1,0) {$x_{11}$} ; 
\node [left] at (-0.9,-0.2) {$x_{1k_1}$} ; 
\node [left] at (-0.6,-0.5) {$x_{21}$} ; 
\node [below right] at (-0.4,-0.6) {$x_{2k_2}$} ; 
\node [below right] at (0.9,-0.2) {$x_{n1}$} ; 
\node [right] at (1,0) {$x_{nk_n}$} ; 
\draw [blue] (0,0) -- (-1,0) ;
\draw [blue] (0,0) -- (-0.9, -0.2) ;
\draw [blue] (0,0) -- (-0.4,-0.6) ;
\draw [blue] (0,0) -- (-0.6,-0.5) ;
\draw [blue] (0,0) -- (0.9,-0.2) ;
\draw [blue] (0,0) -- (1,0) ;
\draw [dotted,thick] (-0.8,0) -- (-0.8,-0.15) ;
\draw [dotted,thick] (-0.48,-0.4) -- (-0.4,-0.5);
\draw [dotted,thick] (0.8,-0.15) -- (0.85, -0.03);
\draw [dotted,thick] (-0.0,-0.25) -- (0.3,-0.2);
\node [right] at (0,0.5) {$e$};
\node [above] at (-0.5,0) {$f_{11}$};
\node [right] at (-0.55,-0.1) {$f_{1k_1}$};
\node [left] at (-0.3,-0.3) {$f_{21}$};
\node [right] at (-0.2,-0.3) {$f_{2k_2}$};
\node [below] at (0.45,-0.1) {$f_{n1}$};
\node [above] at (0.5,0) {$f_{nk_n}$};
\node at (0,-1) {\parbox{0.3\linewidth}{\subcaption{$\st(u,X_s)$}}};
\end{scope}

\begin{scope}
\filldraw (-0.25,0) circle (1pt) ;
\draw [red] (-0.25,0) -- (0,1) ;
\draw [red] (0,0) -- (0,1) ;
\draw [red] (0.4,0) -- (0,1) ;
\filldraw (0,0) circle (1pt) ;
\filldraw (0.4,0) circle (1pt) ;
\filldraw (0,1) circle (1pt) ;
\filldraw (-1,0) circle (1pt) ;
\filldraw (-0.9,-0.2) circle (1pt) ;
\filldraw (-0.4,-0.6) circle (1pt) ;
\filldraw (-0.6,-0.5) circle (1pt) ;
\filldraw (1.2,-.2) circle (1pt) ;
\filldraw (1.3,0) circle (1pt) ;
\node [above left] at (-0.25,0) {$u_1$} ; 
\node [below right] at (0,0) {$u_2$};
\node [above left] at (0.4,0) {$u_n$};
\node [right] at (0,1) {$v'$} ; 
\node [above left] at (-1,0) {$x'_{11}$} ; 
\node [left] at (-0.9,-0.2) {$x'_{1k_1}$} ; 
\node [left] at (-0.6,-0.5) {$x'_{21}$} ; 
\node [below right] at (-0.4,-0.6) {$x'_{2k_2}$} ; 
\node [below right] at (1.2,-0.2) {$x'_{n1}$} ; 
\node [right] at (1.3,0) {$x'_{nk_n}$} ; 
\draw [blue] (-0.25,0) -- (-1,0) ;
\draw [blue] (-0.25,0) -- (-0.9, -0.2) ;
\draw [blue] (0,0) -- (-0.4,-0.6) ;
\draw [blue] (0,0) -- (-0.6,-0.5) ;
\draw [blue] (0.4,0) -- (1.2,-0.2) ;
\draw [blue] (0.4,0) -- (1.3,0) ;
\draw [dotted, thick] (-0.8,0) -- (-0.8,-0.15) ;
\draw [dotted, thick] (-0.48,-0.4) -- (-0.4,-0.5) ;
\draw [dotted, thick] (0.8,-0.10) -- (0.85, -0.03) ;
\draw [dotted,thick] (0.1,-0.25) -- (0.6,-0.2);
\node [right] at (0,0.5) {$e_2$};
\node [left] at (-0.12,0.5) {$e_1$};
\node [right] at (0.2,0.5) {$e_n$} ;
\node [above] at (-0.7,0) {$f'_{11}$};
\node [right] at (-0.75,-0.2) {$f'_{1k_1}$};
\node [above] at (-0.3,-0.3) {$f'_{21}$};
\node [right] at (-0.2,-0.3) {$f'_{2k_2}$};
\node [right] at (0.7,-0.2) {$f'_{n1}$};
\node [above] at (0.8,0) {$f'_{nk_n}$};
\node at (.4,-1) {\parbox{0.3\linewidth}{\subcaption{$T_u$}}};
\end{scope}
\end{tikzpicture}
\caption{Opening $\st(u,X_s)$ to $T_u$} \label{fig_opening_X_s}
\end{center}
\end{figure}

\begin{defn} 
The \emph{opening of $\st(u,X_s)$} along $e$ is a tree $T_u$ (\cref{fig_opening_X_s}) defined as follows: There is one `primary' vertex $v'$ out of which emit $n$ edges $e_1,\cdots ,e_n$, one for each $C_i$. For each $i$, we label the other endpoint of $e_i$ as $u_i$. From each $u_i$, we have $k_i$ edges to the vertices $x'_{i1},\cdots ,x'_{ik_i}$ (compare with $\st(u,X_s)$). We label these edges as $f'_{i1},\cdots ,f'_{ik_i}$. 
The \emph{opened-up graph of $X_s$} along $e$ is a graph $X'_s$ obtained by replacing $\mathring{\st}(u,X_s)$ in $X_s$ by $T_u$, with the obvious identifications. 
\end{defn}

Clearly, $X'_{s}$ is connected, $X_{s} \setminus \mathring{\st}(u,X_s) \hookrightarrow X'_{s}$ and $T_u \hookrightarrow X'_{s}$. There is a natural surjective map from $X'_s$ to $X_s$ which sends each $e_i$ in $T_u$ to $e$. Note that the graphs $X'_{s}$ and $X_s$ are homotopy equivalent. 

\subsection*{Construction}
We now construct a new tubular graph of graphs $X'$ with the same underlying graph $\Gamma_X$ as $X$ and the only change is that $X'_s$ replaces $X_s$. An attaching map of an edge graph is unchanged if $u$ is not in the image, as $X_s \setminus \mathring{\st}(u,X_s)$ embeds in $X'_s$. If $u$ is in the image, we do the obvious modification (see \cref{fig_modifying_attaching_map} for an illustration).  

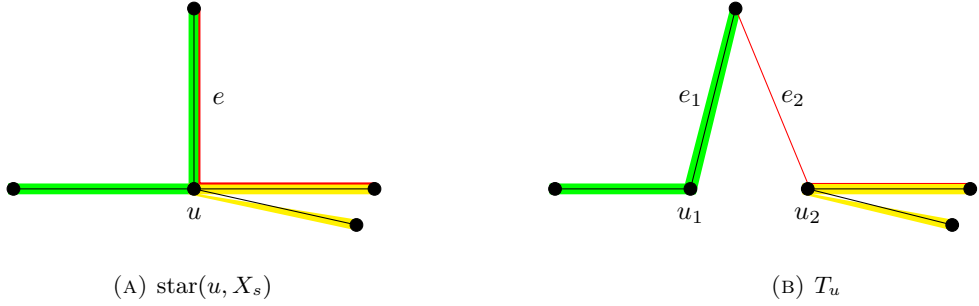
\begin{figure}
\begin{center}
\begin{tikzpicture} [scale = 2.4]
\begin{scope}[xshift=-2.5cm]
\fill [green] (-.03,1) to (.03,1) to (.03,-0.03) to (-1.03,-0.03) to (-1.03,0.03) to (-.03,0.03) to (-.03,1);
\fill [yellow] (1,.03) to (1,-0.03) to (0.03,-0.03) to (0.9,-0.17) to (0.9, -0.23) to (-0.03,-0.03) to (-0.03,0.03) to (1,0.03);
\draw[thick, red] (1,.03) to (.03,.03) to (.03,1);

\filldraw (0,0) circle (1pt) ;
\draw (0,0) -- (0,1) ;
\filldraw (0,1) circle (1pt) ;
\filldraw (-1,0) circle (1pt) ;
\filldraw (0.9,-.2) circle (1pt) ;
\filldraw (1,0) circle (1pt) ;

\draw (0,0) -- (-1,0) ;
\draw (0,0) -- (0.9,-0.2) ;
\draw (0,0) -- (1,0) ;
\node[below] at (0,-0.05) {$u$};
\node[right] at (0.05,0.5) {$e$};

\node at (0,-.5) {\parbox{0.3\linewidth}{\subcaption{$\st(u,X_s)$}}};
\end{scope}

\begin{scope}
\fill [green] (-.03,1) to (.03,1) to (-.22,-0.03) to (-1.03,-0.03) to (-1.03,0.03) to (-.28,0.03) to (-.03,1);
\fill [yellow] (1.3,.03) to (1.3,-0.03) to (0.43,-0.03) to (1.2,-0.17) to (1.2, -0.23) to (0.37,-0.03) to (0.37,0.03) to (1.3,0.03);

\filldraw (-0.25,0) circle (1pt) ;
\draw (-0.25,0) -- (0,1) ;
\draw [red] (0,1) to (0.4,0.03) to (1.3,0.03);

\filldraw (0.4,0) circle (1pt) ;
\filldraw (0,1) circle (1pt) ;
\filldraw (-1,0) circle (1pt) ;
\filldraw (1.2,-.2) circle (1pt) ;
\filldraw (1.3,0) circle (1pt) ;

\node[below] at (-0.25,-0.05) {$u_1$};
\node[below] at (0.4,-.05) {$u_2$};
\node[left] at (-.13,.5) {$e_1$};
\node[right] at (.2,.5) {$e_2$};

\draw  (-0.25,0) -- (-1,0) ;
\draw  (0.4,0) -- (1.2,-0.2) ;
\draw  (0.4,0) -- (1.3,0) ;

\node at (.4,-.5) {\parbox{0.3\linewidth}{\subcaption{$T_u$}}};
\end{scope}
\end{tikzpicture}
\caption{A highlighted path of each colour indicates a part of the image of an attaching map} \label{fig_modifying_attaching_map}
\end{center}
\end{figure}

\begin{notation} The tubular graph of graphs $X'$ is called an \emph{SL-complex} (simplified link complex) of $X$, simplified link in the sense that the vertex $u$ has been replaced by $u_1, \cdots, u_n$ where for each $i$, $\link(u_i)$ is simpler than $\link(u)$.
\end{notation}

There exists a natural map from $X'$ to $X$. Further,

\begin{prop} The tubular graphs of graphs $X$ and $X'$ are homotopy equivalent. \qed
\end{prop}

Since the number of edges of $X'_{s}$ is strictly greater than the number of edges of $X_{s}$, we have
\begin{lemma} \label{lemma_opening_not_isomorphic}
$X'$ is not isomorphic to $X$ as square complexes. \qed
\end{lemma}

\begin{lemma} \label{lemma_no_squares_equal}
Every edge of $X'$ belongs to at least one square and the number of squares in $X'$ is the same as the number of squares in $X$.
\end{lemma}
\begin{proof}
The number of squares in $X'$ is equal to the total number of edges in the cyclic edge graphs, which is equal to the number of squares of $X$.

If a vertical edge does not belong to $T_u$, then since the attaching maps are unchanged from $X$, the vertical edge belongs to at least two squares. Similarly, each $f'_{ij}$ belongs to at least two squares. The edge $e_i$ belongs to a square if and only if a pair of adjacent edges $y_{\lambda_1}, y_{\lambda_2}$ in some edge graph is mapped to a pair $e, f_{ij}$. Such a pair exists as there is an edge between $e$ and ${f_{ij}}$ in $\link(u)$.
\end{proof}

\begin{remark} \label{rmk_no_rudimentary} It is possible that an edge $e_i$ incident to a $u_i$ in $T_u$ is of thickness one in $X'$. However, $e_i$ cannot be rudimentary as $X'_{s}$ is not a circle. In that case, by \cref{prop_thickness1}, $\wt{X}$ is not one-ended.
\end{remark}

\section{The Algorithm}

\begin{defn}
A tubular graph of graphs is \emph{wedge-like} if
\begin{enumerate}
\item it has no hanging trees or rudimentary edges, and
\item either there exists a vertex whose link is not connected or there exists an edge of thickness one.
\end{enumerate}  
\end{defn}

\begin{remark} \label{rmk_wedge_not_one_ended} By \cref{prop_thickness0}, \cref{prop_thickness1} and \cref{prop_1ended_is_BM1}, the fundamental group of a  wedge-like tubular graph of graphs is not one-ended. 
\end{remark}

\begin{theorem} [Main Theorem; \cref{thm_ends_main_intro}] \label{thm_ends_main}
There is an explicit algorithm of polynomial time complexity that takes a tubular graph of graphs as input and returns a homotopy equivalent tubular graph of graphs which is either Brady-Meier, or is a point, or is wedge-like.
\end{theorem}
\begin{proof}
We will prove the theorem by constructing the algorithm. Let $X = X_0$ be the input tubular graph of graphs. Let $k \in \mathbb{N} \cup \{0\}$. 
\begin{enumerate}
\item[Step 1] Check if $X_k$ has hanging subtrees. If yes, collapse each hanging subtree to a point and call the new complex also as $X_k$. Go to the next step.

\item[Step 2] Check if $X_k$ has a rudimentary edge. If yes, remove tubes attached to rudimentary edges (\cref{lemma_removing_rudimentary_edges}) and call the resulting complex also as $X_k$. Go to the next step.

\item [Step 3] Check if $X_k$ has at least one square. If yes, go to the next step. Otherwise, $X_k$ is either a point or wedge-like. Stop.

\item[Step 4] Check if $X_k$ has an edge of thickness zero or one. If yes, $X_k$ is wedge-like. Stop. If not, go to the next step.

\item[Step 5] Check if the link of every vertex of $X_k$ is connected. If yes, go to the next step. If not, $X_k$ is wedge-like. Stop. 

\item[Step 6] Check if $X_k$ satisfies \ref{BM2}. If yes, $X_k$ is Brady-Meier. Stop. If not, go to the next step.

\item[Step 7] Replace $X_k$ by $X_{k+1} = X_k'$, an SL-complex of $X_k$, and go to Step 4. 
\end{enumerate}

We observe that for $l \gneq k \geq 1$, 
\begin{itemize}
\item[(i)] $X_k \ncong X_l$, as each opening increases the number of edges (\cref{lemma_opening_not_isomorphic}).
\item[(ii)] $X_k$ and $X_l$ have the same number of squares (\cref{lemma_no_squares_equal}). 
\item[(iii)] There is no edge of thickness zero in $X_k$ (\cref{lemma_no_squares_equal}).
\end{itemize} 
(i) implies that the procedure above does not return a tubular graph of graphs from an earlier step. Since there are only a finite number of connected square complexes with a fixed number of squares (ii) and no thickness zero edges (iii), the procedure stops in finite time.

Checking if a graph has hanging trees takes linear time in the number of vertices and edges of $X$. Similarly, checking for edges of thickness zero or one or for rudimentary edges takes linear time in the number of edges and squares of $X$. Thus steps 1 through 4 run in linear time in the number of vertices, edges and squares of $X$.

From step 5 onwards, the number of vertices and edges of $X_k$ is bounded by the number of squares of $X_k$, as each edge is contained in a square. 
Steps 5 and 6 run in polynomial (quadratic) time in the number of squares: indeed, the size of a vertex link in $X_k$ is bounded by the number of squares of $X_k$ and checking for connectedness and disconnecting vertices in a vertex link is linear in the number of vertices and edges of the link (see \cite{hopcroft_connected_graph} for details). 

If $n$ is the number of squares in $X$, we claim that the number of times the algorithm goes back to step 4 is at most $n$.
Indeed, the algorithm performs the $k^{th}$ opening-up only if every square of $X_{k-1}$ is of thickness at least two. 
When each edge is of thickness at least two, the number of vertical edges (as well as horizontal edges) of a tubular graph of graphs can be at most equal to the number of squares. 
Observe that the opening procedure in Step 7 increases the number of vertical edges of $X_k$ by at least one, while decreasing the thickness of certain vertical edges. Thus, the algorithm continues at most until each vertical edge is contained in exactly two squares. 
\end{proof}

As an immediate consequence, we have:

\begin{cor}[\cref{cor_grushko_intro}]\label{thm_actual_ends_algo}
There is an algorithm of polynomial time complexity which takes as input a tubular graph of graphs and returns a homotopy equivalent tubular graph of graphs obtained by gluing together certain vertices of a finite collection of Brady-Meier tubular graphs of graphs and finite graphs. 
Further, the free product decomposition induced by cutting along the glued vertices is the Grushko decomposition of the fundamental group of the input tubular graph of graphs. 
\end{cor}

\begin{proof}
Let $X$ be the input tubular graph of graphs with fundamental group $G$. Apply the algorithm of \cref{thm_ends_main} to $X$.
Let $X_N$ be the output. If $X_N$ is a point, then $G$ is trivial. If $X_N$ is Brady-Meier, $G$ has trivial Grushko decomposition. 
If $X_N$ is wedge-like, $G$ is a free product (\cref{rmk_wedge_not_one_ended}). We first remove every edge of thickness one in $X_N$ by the procedure of \cref{fig: thickness one edge}.

Cut $X_N$ along an edge of thickness zero or a locally disconnecting vertex (in the obvious way). Either we get a connected tubular graph of graphs $X'_1$ or we get a disconnected space with components $X'_1, \cdots, X'_n$, where each $X'_i$ is a tubular graph of graphs. In the first case, $G = G_1 * \mathbb{Z}$. In the latter case, $G = G_1* \cdots *G_n$. Apply the algorithm again to each $X'_i$. If each $G_i$ is one-ended, we are done. Otherwise, cut again at an $X'_i$ with a many-ended $G_i$ and repeat.
This procedure terminates in polynomial time. Indeed, at each step we get tubular graphs of graphs whose total number of squares is bounded by the number of squares of $X$.  
\end{proof}

\begin{remark} \label{thm_stallings_algorithm}
We point out that we do not use Stallings' theorem on ends for our proof. In fact, our procedure yields an alternate proof of Stallings' theorem for fundamental groups of tubular graphs of graphs. Similarly, we do not assume the existence of the Grushko decomposition either. Our algorithm proves its existence for the groups under consideration.
\end{remark}

By the uniqueness of the Grushko decomposition, we have:
\begin{cor}[\cref{cor_subadditivity_intro}]\label{cor_subadditivity_sphere_theorem}
Let $X$ be a tubular graph of graphs with fundamental group $G$. Suppose that $G$ admits a free splitting as $G = A * B$. Then there exist tubular graphs of graphs $X_1$ and $X_2$ such that $A$ and $B$ are fundamental groups of $X_1$ and $X_2$ respectively. Moreover, $X_1$ and $X_2$ can be so chosen such that the total number of squares in $X_1$ and $X_2$ is bounded by the number of squares in $X$. \qed
\end{cor}

\section{Whitehead graphs and separability} \label{section_separability}

The goal of this section is to give an alternative proof of Stallings' algorithm to detect whether a finite set of words in a free group is separable (\cref{thm_separability_algo}). We also give a new proof of Whitehead's cut vertex theorem (\cref{prop_stallings_separability}).

Let $F_n$ be a free group of rank $n \geq 2$ and let $W$ be a finite set of non-trivial elements of $F_n$.

\begin{defn}[\cite{stallings_whitehead_graphs}] $W$ is \emph{separable} if there exists a non-trivial free splitting of $F_n = H*K$ such that each element of $W$ is either a conjugate of an element of $H$ or a conjugate of an element of $K$.
\end{defn}

In \cite{stallings_whitehead_graphs}, Stallings developed an algorithm that detects whether $W$ is separable. We will use \cref{thm_ends_main} to obtain an alternative algorithm:

\begin{cor}[Stallings; \cref{cor_separability_intro}] \label{thm_separability_algo}
There exists an algorithm of polynomial time complexity that detects whether a given finite set of words in a finite rank free goup is separable.
\end{cor}

Our method is closely related to Stallings', which uses Whitehead graphs \cite{whitehead} (defined below).
Let $H_n$ denote the orientable 3 dimensional handlebody of genus $n$. Fix an identification of $F_n$ with the fundamental group of $H_n$. A basis $B$ of $F_n$ corresponds to a system of embedded disks $D = \{d_1, \cdots, d_n\}$ such that for an element $b_i \in B$, $b_i$ is represented by a closed path in $H_n$ which starts from the chosen basepoint, intersects $d_i$ transversely and returns to the basepoint without touching any other $d_j$.
Cutting open $H_n$ along these disks results in a 3-ball with 2n disks $d_i^{\pm}$ (such that the chosen representative $b_i$ enters along $d_i^+$ and leaves along $d_i^-$). 
$W$ is represented by a set of curves in $H_n$. After cutting, the set of curves is now a set of arcs between these discs.

\begin{defn}[\cite{whitehead}] The \emph{Whitehead graph} $\Gamma_{F_n,B}(W)$ is the graph with $2n$ vertices labelled $\{b_1^{\pm}, \cdots, b_n^{\pm}\}$ and an edge between two vertices $b_i^+$ (respectively, $b_i^-$) and $b_j^+$ ($b_i^-$) for every arc coming from $W$ between the corresponding discs $d_i^+$ ($d_i^-$) and $d_j^+$ ($d_j^-$) in the cut up handlebody. 
\end{defn}

\cref{fig_whitehead_graph_example} illustrates an example when $n=2$, $B = \{b_1,b_2\}$ and $W = \{b_1b_2b_1\}$.

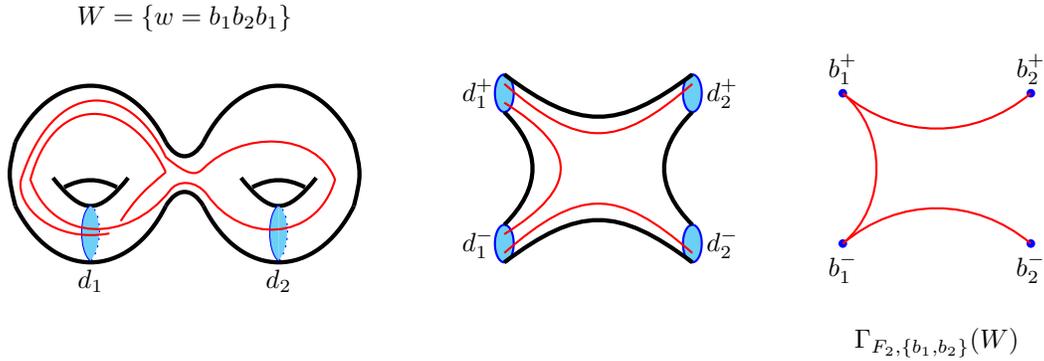
\begin{figure}
\begin{tikzpicture}[scale = 0.43]
\draw [ultra thick] (0,1).. controls (1,-1) and (3,-1).. (4,1) ;
\draw [ultra thick] (0,2.7).. controls (1,4.7) and (3,4.7).. (4,2.7) ;
\draw [ultra thick] (0,1).. controls (-.2,1.8) and (-.2,1.9).. (0,2.7);
\draw [ultra thick] (5,1).. controls (6,-1) and (8,-1).. (9,1) ;
\draw [ultra thick] (5,2.7).. controls (6,4.7) and (8,4.7).. (9,2.7) ;
\draw [ultra thick] (9,1).. controls (9.2,1.8) and (9.2,1.9).. (9,2.7);
\draw [ultra thick] (4,1).. controls (4.3,1.5) and (4.7,1.5).. (5,1);
\draw [ultra thick] (4,2.7).. controls (4.3,2.2) and (4.7,2.2).. (5,2.7);
\draw [ultra thick] (1,1.75) .. controls (1.8,.75) and (2.2,.75).. (3,1.75);
\draw [ultra thick] (6,1.75) .. controls (6.8,.75) and (7.2,.75).. (8,1.75);
\draw [ultra thick] (1.3,1.5) .. controls (1.8,1.75) and (2.2,1.75).. (2.7,1.5);
\draw [ultra thick] (6.3,1.5) .. controls (6.8,1.75) and (7.2,1.75).. (7.7,1.5);
\draw [thick, blue] (2,-.45).. controls (1.7,-0.2) and (1.7,.8).. (2,1);
\draw [thick, blue, dotted] (2,-.45).. controls (2.3,-.1) and (2.3,.8).. (2,1);
\draw [thick, blue] (7,-.45).. controls (6.7,-0.2) and (6.7,.8).. (7,1);
\draw [thick, blue, dotted] (7,-.45).. controls (7.3,-.1) and (7.3,.8).. (7,1);
\fill [cyan!50!white] (7,-.45).. controls (6.7,-0.2) and (6.7,.8).. (7,1);
\fill [cyan!50!white] (7,-.45).. controls (7.3,-.1) and (7.3,.8).. (7,1);
\fill [cyan!50!white] (2,-.45).. controls (1.7,-0.2) and (1.7,.8).. (2,1).. controls (2.3,.8) and (2.3,-.1).. (2,-.45);
\node [below] at (2,-.45) {$d_1$};
\node [below] at (7,-.45) {$d_2$};
\node at (4.5,6) {$W = \{w = b_1b_2 b_1\}$};

\draw [thick, red] (4,1.9) .. controls (3,4) and (1,4).. (0.4,1.7).. controls (1,0) and (3,0).. (4.1,1.4).. controls (4.3,1.7) and (4.8,1.7).. (5.1,1.3).. controls (6,0) and (8,0).. (8.5,1.7).. controls (8,3) and (6,3).. (5,2).. controls (4.8,1.8) and (4.6,1.8).. (4,2.3).. controls (3,4.4) and (1,4.4).. (0.15,1.7).. controls (0,1) and (1,0).. (2.5,.27);
\draw [thick, red] (2.8,.6).. controls (3.2,1.2) and (3.4,1.4).. (4,1.9);

\draw [thick, blue, fill = cyan!50!white] (13,4) ellipse (0.25 and 0.5);
\draw [thick, blue, fill = cyan!50!white] (13,0) ellipse (0.25 and 0.5);
\draw [thick, blue, fill = cyan!50!white] (18,4) ellipse (0.25 and 0.5);
\draw [thick, blue, fill = cyan!50!white] (18,-0) ellipse (0.25 and 0.5);
\draw [ultra thick] (13,3.5).. controls (14,2.5) and (14,1.5).. (13,0.5);
\draw [ultra thick] (18,3.5).. controls (17,2.5) and (17,1.5).. (18,0.5);
\draw [ultra thick] (13,4.5).. controls (15,3) and (16,3).. (18,4.5);
\draw [ultra thick] (13,-.5).. controls (15,1) and (16,1).. (18,-.5);
\draw [thick, red] (13,4.25).. controls (15,2.5) and (16,2.5).. (18,4.25);
\draw [thick, red] (13,-.25).. controls (15,1.5) and (16,1.5).. (18,-.25);
\draw [thick, red] (13,3.75).. controls (15,2.5) and (15,1.5).. (13,0.25);
\node  at (12.3,4) {$d_1^+$};
\node  at (12.3,0) {$d_1^-$};
\node  at (18.8,4) {$d_2^+$};
\node  at (18.8,0) {$d_2^-$};

\draw [fill, blue] (22,4) circle (0.1);
\draw [fill, blue] (27,4) circle (0.1);
\draw [fill, blue] (22,0) circle (0.1);
\draw [fill, blue] (27,0) circle (0.1);
\draw [thick,red] (22,4) to [out = 320, in = 220] (27,4);
\draw [thick,red] (22,0) to [out = 40, in = 140] (27,0);
\draw [thick,red] (22,4) to [out = 320, in = 40] (22,0);
\node [above] at (22,4) {$b_1^+$};
\node [below] at (22,0) {$b_1^-$};
\node [above] at (27,4) {$b_2^+$};
\node [below] at (27,0) {$b_2^-$};
\node [below] at (24.5,-2) {$\Gamma_{F_2,\{b_1,b_2\}}(W)$};

\end{tikzpicture}
\caption{A Whitehead graph}
\label{fig_whitehead_graph_example}
\end{figure}

Recall that if $Y$ is a topological space, then a \emph{cut point} $y \in Y$ is a point such that $Y \setminus \{y\}$ is not connected.
There is a well-known result about the separability of $W$. 
\begin{prop}[\cite{whitehead}] \label{prop_stallings_separability} If $W$ is separable, then for any basis $B$, the Whitehead graph $\Gamma_{F_n,B}(W)$ is either disconnected, or has a cut vertex.
\end{prop} 
More details can be found in \cite{stallings_whitehead_graphs}. We will re-prove \cref{prop_stallings_separability} above using tubular graphs of graphs. 
Stallings constructs his algorithm to detect separability by choosing a Whitehead automorphism whenever there is a cut vertex in a Whitehead graph. Our strategy is to use the machinery of \cref{thm_ends_main} when a Whitehead graph contains a cut vertex.  

\subsection{Construction of a double} \label{subsection_double} Let $R_n$ denote an oriented rose with petals $\{a_1, \cdots, a_n\}$. Fix an identification of $F_n$ with the fundamental group  of $R_n$ such that each petal of $R_n$ in the positive direction represents a distinct element of the basis $B = \{b_1, \cdots, b_n\}$. For each element $w_j \in W = \{w_1, \cdots, w_k\}$, let $\phi_j : C_j \to R_n$ denote a cycle from the circle $C_j$ such that $\phi_j$ induces the word $w_j$ in $F_n$. We assume that $w_j$ is cyclically reduced so that $\phi_j$ is an immersion. Subdivide $R_n$ and each $C_j$ so that each $\phi_j$ is a simplicial immersion between simplicial graphs. Denote the subdivided $R_n$ by $X_s$.

We call the descendant $v$ in $X_s$ of the unique vertex of $R_n$ as the \emph{special vertex} of $X_s$.

We define the \emph{double} of $X_s$ along $W$ to be the tubular graph of graphs $X$ such that $X_s$ is a vertical graph of $X$ with exactly $k$ tubes attached to $X_s$ via the attaching maps $\phi_j$. Further, the underlying graph of $X$ consists of exactly two vertices and $k$ edges between them, where each vertex graph is isomorphic to $X_s$ and each edge graph is isomorphic to a unique $C_j$ with attaching map $\phi_j$ on both sides.

\begin{lemma} \label{lemma_links_special_vertices_imp}
The vertex link of every vertex of $X_s \subset X$ is connected with no cut vertices if and only if the link of the special vertex $v$ is connected with no cut vertices.
\end{lemma}

\begin{proof}
The main case to consider is that of a vertex $u \neq v \in X_s$ whose vertex link is either not connected or has a cut vertex. Then there exists at most one path between the two vertical vertices ${e_1},{e_2}$ in $\link(u)$. 
There is a bijective correspondence between paths in $\link(u)$ between $e_1$ and $e_2$ and paths in $\link(v)$ between vertices that induce the petal in $R_n$ that contains $u$. Hence the result.
\end{proof}

The main lemma of this section is the following:

\begin{lemma} \label{lemma_links_whitehead_graphs} The link in $X$ of the special vertex $v$ is isomorphic as graphs to the first subdivision of the Whitehead graph $\Gamma_{F_n,B}(W)$.
\end{lemma}
\begin{proof}
Each petal of $R_n$ induces two vertical vertices in $\link(v)$  and hence there are $2n$ vertical vertices.
Paths of length two in $\link(v)$ between these vertices correspond to the occurrence of the respective letters in a word $w_j$ of $W$. The isomorphism is then clear.
\end{proof}

\begin{proof}[Proof of \cref{prop_stallings_separability}]
Let $X$ be the double of ($X_s,W)$, with fundamental group $G$. $G$ is not one-ended as $W$ is separable. By \cref{prop_1ended_is_BM1}, the link of some vertex of $X_s$ is either not connected or has a cut vertex. \cref{lemma_links_special_vertices_imp} and \cref{lemma_links_whitehead_graphs} then give the result.
\end{proof}

We will need the following result by Wilton \cite[Theorem 18]{wilton_one-ended_groups}:

\begin{theorem}[Wilton] \label{thm_wilton_ends}
The fundamental group of a graph of free groups with cyclic edge groups is freely indecomposable if and only if every vertex  group is freely indecomposable relative to the incident edge groups.
\end{theorem}

Note that a vertex group is freely indecomposable relative to the incident edge groups if and only if the set of words induced by the generators of these edge groups is not separable.

\begin{proof}[Proof of \cref{thm_separability_algo}] Let $W$ be the given set of words of the free group $F$. Let $X$ be the double of $(X_s,W)$ and $G$ its fundamental group. Apply the algorithm of \cref{thm_actual_ends_algo} to detect whether $G$ is one-ended. By \cref{thm_wilton_ends} above, $G$ is one-ended if and only if $W$ is not separable.
\end{proof} 

\section{The general case} \label{section_general}

We recall that by a \emph{graph of groups} $\mathcal{G}$, we mean the following data: $\Gamma$ is a connected graph, called the underlying graph. For each vertex $s$ (edge $a$) of $\Gamma$, $G_s$ ($G_a$) is a group. Further, whenever $a$ is incident to $s$, $\partial_{a,s} : G_{a} \to G_{s}$ is an injective homomorphism. Given a graph of groups as above, one can naturally associate with it a graph of spaces $\mathcal{X}_{\mathcal{G}}$ with the same underlying graph $\Gamma$ such that for each vertex $s$ (edge $a$) of $\Gamma$, $X_a$ ($X_s$) is a connected topological space such that $\pi_1(X_a) \cong G_a$ ($\pi_1 (X_s) \cong G_s$). The \emph{fundamental group} of the graph of groups $\mathcal{G}$ is the fundamental group of the geometric realisation of $\mathcal{X}_{\mathcal{G}}$ (\cref{section_graphs_spaces}). We refer the reader to \cite{scottandwall} for details. In this section, we freely switch between graphs of groups and graphs of spaces as required.

\begin{theorem}[\cref{thm_grushko_general_intro}] \label{thm_grushko_general}
There is an algorithm of polynomial time complexity which takes a graph of free groups with cyclic edge groups as input and returns the Grushko decomposition of its fundamental group.
\end{theorem}

Before going into the proof, we recall the definitions of collapse and blow up. Given a graph of groups $\mathcal{G}$, and an edge $a$ of the underlying graph $\Gamma$ with endpoints $s_1,s_2$, by a \emph{collapse} of the edge $s$ we mean a graph of groups $\mathcal{G}'$ with the following data: the underlying graph $\Gamma'$ is a quotient of $\Gamma$ with the edge $s$ collapsed to a point (vertex) $a_s$. The only change in vertex (edge) groups is the new vertex group $G_{a_s}$. $G_{a_s}$ is the fundamental group of the graph of groups whose underlying graph is the edge $s$, the vertex groups are $G_{s_1}$ and $G_{s_2}$ respectively and the edge group is $G_a$, with the injective homomorphisms from the edge group remaining the same as in $\mathcal{G}$. 

An \emph{elementary blow up} is the inverse operation of a collapse of an edge. A \emph{blow up} is an iterated process of finitely many elementary blow up operations. Note that a collapse or blow up operation does not change the fundamental group.

We will also recall the notion of a relative Grushko decomposition. Let $F$ be a finite rank free group and let $W$ be a finite set of words in $F$. The \emph{relative Grushko decomposition} of the pair $(F,W)$ is a free splitting of $F$ such that each element of $W$ conjugates into a free factor of the splitting. Further, each free factor of the splitting is itself freely indecomposable relative to $W$.

\begin{proof}
Given a graph of free groups with cyclic edge groups $\mathcal{G}$, for each vertex $s$ of the underlying graph $\Gamma$, we have a pair $(G_s,W_s)$, where $G_s$ denotes the free vertex group at $s$ and $W_s$ is the set of words induced by generators of the incident edge groups. 
The Grushko decomposition of the fundamental group of $\mathcal{G}$ can be obtained in two steps: 

\begin{enumerate}
    \item Obtain for each pair $(G_s,W_s)$ its \emph{relative Grushko decomposition} $\mathcal{G}^{Gr}_s$ in the following way: For the pair $(G_s, W_s)$, let $X_s$ denote a suitable subdivision of an oriented rose with fundamental group $G_s$ (see \cref{subsection_double}) and let $Y_s$ denote its double. Apply \cref{thm_actual_ends_algo} to obtain a tubular graph of graphs $Y^{{Gr}}_s$ which induces the Grushko decomposition of $\pi_1(Y_s)$. Note that by \cref{thm_wilton_ends}, any free splitting of $\pi_1(Y_s)$ is a free splitting of $G_s$ relative to $W_s$. Thus the obtained Grushko decomposition of $\pi_1(Y_s)$ is a double over $W_s$ of the relative Grushko decomposition $\mathcal{G}^{Gr}_s$ of $(G_s,W_s)$. 
    Let $\Gamma_s$ denote the underlying graph of $\mathcal{G}^{Gr}_s$. 
    
    \item Modify $\mathcal{G}$ by blowing-up each vertex $s$ to $\mathcal{G}^{{Gr}}_s$. Let $\mathcal{G}'$ be the new graph of groups. Note that the edge groups of $\mathcal{G}'$ are either cyclic or trivial (trivial if and only if the edge belongs to some $\Gamma_s$).
    Collapse all edges with nontrivial edge group to obtain a new graph of groups $\mathcal{G}^{{Gr}}$. By \cref{thm_wilton_ends}, the fundamental group of $\mathcal{G}$ admits no other free splitting and hence $\mathcal{G}^{{Gr}}$ is the required Grushko decomposition. 
\end{enumerate}
The running time of the above algorithm is of the order of the sum of the running times of the algorithms to obtain the relative Grushko decompositions in step (1). But each such algorithm runs in polynomial time in the number of squares of $Y_s$. The number of squares of $Y_s$ is the sum of the lengths of the words of $W_s$ in $X_s$. Thus, the algorithm runs in polynomial time in the lengths of the words defined by the incident edge groups (in the respective vertex groups) of the input graph of free groups with cyclic edge groups.
\end{proof}

\bibliographystyle{alpha}
\bibliography{Vertex_links_grushko_decomposition}

\begin{thebibliography}{RVW07}

\bibitem[BH99]{bridsonhaefliger}
Martin~R. Bridson and Andr{\'e} Haefliger.
\newblock {\em Metric spaces of non-positive curvature}, volume 319 of {\em
  Grundlehren der Mathematischen Wissenschaften [Fundamental Principles of
  Mathematical Sciences]}.
\newblock Springer-Verlag, Berlin, 1999.

\bibitem[BM01]{bradymeier}
Noel Brady and John Meier.
\newblock Connectivity at infinity for right angled {A}rtin groups.
\newblock {\em Trans. Amer. Math. Soc.}, 353(1):117--132, 2001.

\bibitem[DF05]{feighn_grushko}
Guo-An Diao and Mark Feighn.
\newblock The {G}rushko decomposition of a finite graph of finite rank free
  groups: an algorithm.
\newblock {\em Geom. Topol.}, 9:1835--1880, 2005.

\bibitem[DG08]{dahmani_groves_grushko}
Fran\c{c}ois Dahmani and Daniel Groves.
\newblock Detecting free splittings in relatively hyperbolic groups.
\newblock {\em Trans. Amer. Math. Soc.}, 360(12):6303--6318, 2008.

\bibitem[Fre31]{Freudenthal_ends}
Hans Freudenthal.
\newblock {\"U}ber die enden topologischer r{\"a}ume und gruppen.
\newblock {\em Mathematische Zeitschrift}, 33(1):692--713, 1931.

\bibitem[Ger99]{gerasimov}
V~Gerasimov.
\newblock Detecting connectedness of the boundary of a hyperbolic group.
\newblock Unpublished, 1999.

\bibitem[Gro87]{gromov_hyperbolic}
M.~Gromov.
\newblock Hyperbolic groups.
\newblock In {\em Essays in group theory}, volume~8 of {\em Math. Sci. Res.
  Inst. Publ.}, pages 75--263. Springer, New York, 1987.

\bibitem[Gru40]{grushko}
I.~Gruschko.
\newblock \"uber die {B}asen eines freien {P}roduktes von {G}ruppen.
\newblock {\em Rec. Math. [Mat. Sbornik] N.S.}, 8 (50):169--182, 1940.

\bibitem[Hop44]{hopf_groupends}
Heinz Hopf.
\newblock Enden offener {R}\"aume und unendliche diskontinuierliche {G}ruppen.
\newblock {\em Comment. Math. Helv.}, 16:81--100, 1944.

\bibitem[HT73]{hopcroft_connected_graph}
John Hopcroft and Robert Tarjan.
\newblock Algorithm 447: Efficient algorithms for graph manipulation.
\newblock {\em Commun. ACM}, 16(6):372--378, June 1973.

\bibitem[Jac69]{jaco_free_product}
William Jaco.
\newblock Three-manifolds with fundamental group a free product.
\newblock {\em Bull. Amer. Math. Soc.}, 75:972--977, 1969.

\bibitem[JLR02]{jaco_grushko_jsj_algorithm}
William Jaco, David Letscher, and J.~Hyam Rubinstein.
\newblock Algorithms for essential surfaces in 3-manifolds.
\newblock In {\em Topology and geometry: commemorating {SISTAG}}, volume 314 of
  {\em Contemp. Math.}, pages 107--124. Amer. Math. Soc., Providence, RI, 2002.

\bibitem[MS20]{meda_jsj}
Suraj~Krishna M~S.
\newblock Immersed cycles and the {JSJ} decomposition.
\newblock {\em Algebr. Geom. Topol.}, 20(4):1877--1938, 2020.

\bibitem[Ray60]{raymond}
Frank Raymond.
\newblock The end point compactification of manifolds.
\newblock {\em Pacific J. Math.}, 10:947--963, 1960.

\bibitem[RVW07]{whitehead_minimization}
Abd\'{o} Roig, Enric Ventura, and Pascal Weil.
\newblock On the complexity of the {W}hitehead minimization problem.
\newblock {\em Internat. J. Algebra Comput.}, 17(8):1611--1634, 2007.

\bibitem[Sag95]{sageev}
Michah Sageev.
\newblock Ends of group pairs and non-positively curved cube complexes.
\newblock {\em Proc. London Math. Soc. (3)}, 71(3):585--617, 1995.

\bibitem[Ser80]{serre}
Jean-Pierre Serre.
\newblock {\em Trees}.
\newblock Springer-Verlag, Berlin-New York, 1980.
\newblock Translated from the French by John Stillwell.

\bibitem[Spe49]{specker}
Ernst Specker.
\newblock Die erste {C}ohomologiegruppe von \"{U}berlagerungen und
  {H}omotopie-{E}igenschaften dreidimensionaler {M}annigfaltigkeiten.
\newblock {\em Comment. Math. Helv.}, 23:303--333, 1949.

\bibitem[Sta99]{stallings_whitehead_graphs}
John~R. Stallings.
\newblock Whitehead graphs on handlebodies.
\newblock In {\em Geometric group theory down under ({C}anberra, 1996)}, pages
  317--330. de Gruyter, Berlin, 1999.

\bibitem[SW79]{scottandwall}
Peter Scott and Terry Wall.
\newblock Topological methods in group theory.
\newblock In {\em Homological group theory ({P}roc. {S}ympos., {D}urham,
  1977)}, volume~36 of {\em London Math. Soc. Lecture Note Ser.}, pages
  137--203. Cambridge Univ. Press, Cambridge-New York, 1979.

\bibitem[Tou18]{touikan_two_torsion}
Nicholas W.~M. Touikan.
\newblock Detecting geometric splittings in finitely presented groups.
\newblock {\em Trans. Amer. Math. Soc.}, 370(8):5635--5704, 2018.

\bibitem[Whi36]{whitehead}
J.~H.~C. Whitehead.
\newblock On equivalent sets of elements in a free group.
\newblock {\em Ann. of Math. (2)}, 37(4):782--800, 1936.

\bibitem[Wil12]{wilton_one-ended_groups}
Henry Wilton.
\newblock One-ended subgroups of graphs of free groups with cyclic edge groups.
\newblock {\em Geom. Topol.}, 16(2):665--683, 2012.

\bibitem[Wis96]{wisephd}
Daniel~T. Wise.
\newblock {\em Non-positively curved squared complexes: {A}periodic tilings and
  non-residually finite groups}.
\newblock ProQuest LLC, Ann Arbor, MI, 1996.
\newblock Thesis (Ph.D.)--Princeton University.

\end{thebibliography}

\end{document}